\documentclass[12pt]{amsart}
\usepackage{amsmath}
\usepackage{amssymb}
\usepackage{latexsym}
\usepackage{amscd}
\usepackage{citesort}
\usepackage{graphicx} 
\usepackage{amsthm}
\usepackage{mathrsfs}
\usepackage{xypic}
\usepackage{bm}

\usepackage{amsmath} 
\usepackage{amssymb} 

\usepackage{yfonts}

\newdimen\AAdi%
\newbox\AAbo%
\def\AAk#1#2{\s_etbox\AAbo=\hbox{#2}\AAdi=\wd\AAbo\kern#1\AAdi{}}%
\def\AAr#1#2#3{\s_etbox\AAbo=\hbox{#2}\AAdi=\ht\AAbo\raise#1\AAdi\hbox{#3}}%
\font\tenmsb=msbm10 at 12pt \font\sevenmsb=msbm7 at 8pt
\font\fivemsb=msbm5 at 6pt
\newfam\msbfam
\textfont\msbfam=\tenmsb \scriptfont\msbfam=\sevenmsb
\scriptscriptfont\msbfam=\fivemsb
\def\Bbb#1{{\tenmsb\fam\msbfam#1}}
\newtheorem{thm}{Theorem}[section]

\newtheorem{rem}{Remark}[section]
\newtheorem{pro}{Proposition}[section]

\newcommand{\ba}{\begin{array}}
\newcommand{\ea}{\end{array}}

\newcommand{\Section}[2]{\setcounter{equation}{0}
\allowdisplaybreaks
\section[#1]{#2}}

\def\n{\nabla}

\def\f#1#2{\frac{#1}{#2}}

\def\pd#1#2{\frac {\partial #1}{\partial #2}}

\def\td{\tilde}

\def\a{\alpha}
\def\be{\beta}

\def\p#1{\partial #1}

\def\de{\delta}
\def\De{\Delta}

\def\ep{\varepsilon}
\def\G{\Gamma}
\def\g{\gamma}

\def\la{\lambda}

\def\si{\sigma}
\def\Si{\Sigma}

\def\C{\Bbb{C}}

\def\lan{\langle}
\def\ran{\rangle}
\def\ra{\rightarrow}
\def\sw{\textswab}
\def\pz{\f{\p}{\p z}}
\def\pw{\f{\p}{\p w}}
\def\pdw{\f{\p}{\p \bar{w}}}
\def\ze{\zeta}
\def\Dirac{D\hskip -2.9mm \slash\ }
\def\dirac{\partial\hskip -2.6mm \slash\ }

\begin{document}
\title
[A structure theorem of Dirac-harmonic maps ] {A structure theorem of Dirac-harmonic maps between spheres}

\author
[Ling Yang]{Ling Yang}
\address
{Institute of Mathematics, Fudan University, Shanghai 200433, China
and Key Laboratory of Mathematics for Nonlinear Sciences (Fudan
University), Ministry of Education} \email{051018016@fudan.edu.cn}

\begin{abstract}

For an arbitrary Dirac-harmonic map $(\phi,\psi)$ between compact oriented Riemannian surfaces, we shall study the zeros
of $|\psi|$. With the aid of Bochner-type formulas, we explore the relationship between the order of the zeros
of $|\psi|$ and the genus of $M$ and $N$. On the basis, we could clarify all of nontrivial Dirac-harmonic maps from $S^2$ to $S^2$.

\end{abstract}

\renewcommand{\subjclassname}{%
  \textup{2000} Mathematics Subject Classification}
\subjclass{58E20, 53C27.}
\date{}
\maketitle

\Section{Introduction}{Introduction}

Let $(M,h)$ be an $m$-dimensional Riemannian spin manifold; $Spin M$
denotes the Spin-bundle on $M$, and $\eta:Spin M\ra SO M$ is the bundle map,
where $SO M$ denotes the tangent orthonormal frame bundle on $M$.
Denote by $\Si M$ the spinor bundle associated to
$Spin M$, i.e. $\Si M=Spin M\times_\rho \Si_m$, where $\rho:Spin_m\ra \Si_m$ is the standard
representation. On $\Si M$ we can choose an Hermitian product $\lan\ ,\ \ran$,
such that
\begin{equation}
\lan X\cdot \psi,\xi\ran=-\lan \psi,X\cdot \xi\ran\qquad X\in \G(TM), \psi,\xi\in \G(\Si M).
\end{equation}
Here
\begin{equation}
\sw{m}: X\otimes \psi\mapsto X\cdot \psi
\end{equation}
 is the Clifford multiplication. There is a connection on $\Si M$
induced by the Levi-Civita connection of $SOM$; denote it by $\n$; and it is well known that
$\n$ is compatible with $\lan\ ,\ \ran$.
Let $\phi$ be a smooth map from $M$ to another Riemannian manifold $(N,g)$ of dimension $n\geq 2$. Denote
by $\phi^{-1}TN$ the pull-back bundle of $TN$ and by $\Si M\otimes \phi^{-1}TN$ the twisted bundle.
 On it there is a metric
induced from those on $\Si M$ and $\phi^{-1}TN$. Similarly we have a natural connection
$\td{\n}$ on $\Si M\otimes \phi^{-1}TN$ induced from
those on $\Si M$ and $\phi^{-1}TN$. Based on it, we can define the Dirac operator along the map $\phi$ by
\begin{equation}
D\hskip -2.9mm \slash\ \psi=\sw{m}\circ \td{\n}\psi.
\end{equation}
Here $\psi$ is a smooth section of $\Si M\otimes \phi^{-1}TN$.

In \cite{c-j-l-w1}, Q. Chen, J. Jost, J. Li and G. Wang introduced a functional that couples the nonlinear sigma model with a spinor field:
\begin{equation}
L(\phi,\psi)=\int_M \big[|d\phi|^2+\lan \psi,D\hskip -2.9mm \slash\ \psi\ran\big]*1
\end{equation}
The critical points of the functional is called Dirac-harmonic maps. In the paper, some geometric and analytic aspects of
such maps were studied, especially a removable singularity theorem was established. Later in \cite{c-j-l-w2}, \cite{c-j-w}, \cite{z1} and \cite{z2},
another geometric and analytic properties of Dirac-harmonic maps were studied.

Obviously there are two type of trivial Dirac-harmonic maps. One is $(\phi,0)$, where $\phi$ is a harmonic map, and another
is $(y,\psi)$, where $y$ is a point in $N$ viewed as a constant map, $\psi$ is a harmonic spinor. In \cite{c-j-l-w1},
the authors constructed non-trivial Dirac-harmonic maps $(\phi,\psi)$ from $S^2$ to $S^2$, where $\phi$ is a (possible branched)
conformal map, $\psi$ could be written in the form
\begin{equation}
\psi=e_\a\cdot \Psi\otimes \phi_*(e_\a),
\end{equation}
$\{e_\a:\a=1,2\}$ is a local orthonormal frame field on $S^2$, and $\Psi$ is a twisor spinor. It
is natural to ask whether there exists another form of Dirac-harmonic maps from $S^2$ to $S^2$. And
furthermore, is there a Dirac-harmonic map $(\phi,\psi)$ such that $\phi$ is not a harmonic map?

In the theory of harmonic maps between two compact Riemannian surfaces, Bochner formulas of $\log|\p u|$
and $\log|\bar{\partial}u|$ play an important role ($u$ denotes a harmonic map). From it several interesting
formulas easily follow, which tell us the relationship between the order of the zeros of $\log|\p u|$
and $\log|\bar{\partial}u|$ and the genus of $M$ and $N$; and moreover we can obtain some
uniqueness theorems and non-existence theorems (see \cite{s-y} Chapter I). This phenomenon
motives us to study the zeros of $|\psi|$.

Now we give a brief outline of the paper. In Section \ref{s2},
the subjects we study are general Dirac-harmonic maps. In the viewpoint that $\Si M\otimes \phi^{-1}TN=\Si M\otimes (\phi^{-1}TN)^{\C}$,
$\psi$ could be written as $\psi=\psi^j\otimes W_j$, where $\{W_1,\cdots,W_n\}$ is a local complex tangent frame field on $N$;
and we derive the Euler-Lagrange equation of $L$ by using the above denotation. In Section \ref{s1}, we assume $M$ and $N$
to be oriented Riemannian surfaces; the equations of harmonic spinor $\psi$ along $\phi$ in the local complex coordinates
are derived, which imply that the zeros of $|\pi_1^+(\psi)|$ are isolated, unless $|\pi_1^+(\psi)|$
is identically zero, so are $|\pi_0^+(\psi)|$, $|\pi_1^-(\psi)|$ and $|\pi_0^-(\psi)|$. (Here the definition of
$\pi_1^+,\pi_0^+,\pi_1^-,\pi_0^-$ is introduced in Section \ref{s1}.) In Section \ref{s3}-\ref{s4}, under the further assumption that
$M$ and $N$ are both compact, we derive several Bochner formulas of $\log|\pi_1^+(\psi)|, \log|\pi_0^+(\psi)|,
\log|\pi_1^-(\psi)|, \log|\pi_0^-(\psi)|$ on the basis of Weitzenb\"{o}ck-type formulas
of $\psi$ and furthermore give the proof of main theorems as follows, including a structure theorem of Dirac-harmonic maps from
$S^2$ to $S^2$. (In the process, it is necessary to use the results
in Section \ref{s2}-\ref{s1}.)
\begin{thm}
$M$ and $N$ are both compact oriented Riemannian surfaces, and $(\phi,\psi)$ is a Dirac-harmonic map from $M$ to $N$.
If $g_M=0$ or $|g_M-1|<|\deg(\phi)||2g_N-2|$, then $\phi$ has to be a harmonic map.
\end{thm}

\begin{thm}
If $M=N=S^2$ equipped with arbitrary metric, $(\phi,\psi)$ is a nontrivial Dirac-harmonic map from $M$ to $N$, then
$\phi$ has to be holomorphic or anti-holomorphic, $\psi$ could be written in the form
\begin{equation}
\psi=e_\a\cdot \Psi\otimes \phi_*(e_\a),
\end{equation}
where $\Psi$ is a twistor spinor (possibly with isolated singularities).
\end{thm}

Please note that here and in the sequel we use the summation convention and agree the
range of indices:
$$1\leq i,j,k\leq n;\qquad 1\leq \a,\be\leq m.$$
We refer to \cite{l-m} and \cite{f} for more background material on
spin structures and Dirac operators.

\Section{Euler-Lagrange equations of Dirac-harmonic maps}{Euler-Lagrange equations of Dirac-harmonic maps}\label{s2}

Denote the complexification of $\phi^{-1}TN$ by $(\phi^{-1}TN)^{\C}$. Obviously
$\Si M\otimes \phi^{-1}TN\subset \Si M\otimes (\phi^{-1}TN)^{\C}$. On the other hand, for
any $\psi\in \Si M$, $X+\sqrt{-1}Y\in (\phi^{-1}TN)^\C$ (here $X,Y\in \phi^{-1}TN$),
$$\psi\otimes (X+\sqrt{-1}Y)=\psi\otimes X+\sqrt{-1}\psi\otimes Y\in\Si M\otimes \phi^{-1}TN;$$
which implies $\Si M\otimes (\phi^{-1}TN)^{\C}\subset \Si M\otimes \phi^{-1}TN$.
Hence $\Si M\otimes (\phi^{-1}TN)^{\C}=\Si M\otimes \phi^{-1}TN$. The pull-back metric $\phi^{-1}g$ on $\phi^{-1}TN$
could be naturally extended to a Hermitian
product on $(\phi^{-1}TN)^\C$; and there is a natural Hermitian
product on $\Si M\otimes (\phi^{-1}TN)^{\C}$
 induced from those on $\Si M$ and $(\phi^{-1}TN)^{\C}$, which is also denoted by $\lan\ ,\ \ran$.

For each point $x\in M$, we can choose $\{W_i\in \G(TU):1\leq i\leq n\}$, where $U$ is a neighborhood of $\phi(x)$, such that
\begin{equation}
(T_{\phi(y)} N)^\C=\bigoplus_{i=1}^n \C W_i(\phi(y))\qquad y\in \phi^{-1}(U),
\end{equation}
then on $\phi^{-1}(U)$, $\psi$ could be expressed by
\begin{equation}
\psi(y)=\psi_j(y)\otimes W_i(\phi(y)),
\end{equation}
where $\psi_1,\cdots,\psi_n\in \G(\Si \big(\phi^{-1}(U)\big)$.
We shall derive the Euler-Lagrange equations for $L$ by using the above denotation.

\begin{pro}\label{p1}

Let $\{e_\a:1\leq \a\leq m\}$ be a local tangent orthonormal frame field,
Then the Euler-Lagrange equations for $L$ are
\begin{eqnarray}
D\hskip -2.9mm \slash\ \psi&=&0\label{eq1}\\
\tau(\phi)&=&-\lan \psi^j,e_\a\cdot \psi^k\ran R_{\overline{W}_j,W_k}^N \phi_*(e_\a).\label{eq4}
\end{eqnarray}
\end{pro}

\begin{proof}

At first, we consider a family of $\psi_t$ with $\f{d\psi_t}{dt}=\eta$ at $t=0$ and fix $\phi$. Since $D\hskip -2.9mm \slash\ $
is formally self-adjoint (see \cite{c-j-l-w1}), we have
\begin{equation}
\aligned
\f{dL}{dt}\Big|_{t=0}&=\int_M \lan \eta,D\hskip -2.9mm \slash\ \psi\ran+\lan \psi,D\hskip -2.9mm \slash\ \eta\ran=\int_M \lan \eta,D\hskip -2.9mm \slash\ \psi\ran+\lan D\hskip -2.9mm \slash\ \psi,\eta\ran\\
&=2\int_M Re\lan \eta,D\hskip -2.9mm \slash\ \psi\ran.
\endaligned
\end{equation}
Since $\eta$ could be chosen arbitrarily, (\ref{eq1}) is easily followed.

Now we consider a variation $\phi_t$ ($t\in (-\ep,\ep)$) of $\phi$ such that $\phi_t=\phi$ outside a compact
set $K\subset\subset \phi^{-1}(U)$ and $\phi_t(K)\subset U$; denote $\psi(y)=\psi^j(y)\otimes W_j(\phi(y))$ for each $y\in U$, then we define
$\psi_t(y)=\psi^j(y)\otimes W_j(\phi_t(y))$.  Denote $\xi=\f{d\phi_t}{dt}\big|_{t=0}$.
Obviously
\begin{equation}\label{psi}
\f{dL}{dt}\Big|_{t=0}=\int_M \f{d}{dt}\Big|_{t=0}|d\phi_t|^2+\int_M \f{d}{dt}\Big|_{t=0}\lan \psi,D\hskip -2.9mm \slash\ \psi\ran=I+II,
\end{equation}
and
\begin{equation}\label{psi1}
I=-2\int_M \lan \xi,\tau(\phi)\ran.
\end{equation}
Here $\tau(\phi)$ denotes the tension field of $\phi$.
Since $\psi=\psi^j\otimes W_j$, we have
\begin{equation}
D\hskip -2.9mm \slash\ \psi=\p\hskip -2.6mm \slash\ \psi^k\otimes W_k+e_\a\cdot \psi^k\otimes \n_{e_\a}W_k,
\end{equation}
where $\p\hskip -2.6mm \slash\ $ denotes the usual Dirac operator.
Then
$$\aligned
\f{d}{dt}\Big|_{t=0}D\hskip -2.9mm \slash\ \psi&=\p\hskip -2.6mm \slash\ \psi^k\otimes \n_{\f{\p}{\p t}}W_k+e_\a\cdot \psi^k\otimes \n_{\f{\p}{\p t}}\n_{e_\a}W_k\\
&=e_\a\cdot \n_{e_\a}\psi^k\otimes \n_{\f{\p}{\p t}}W_k+e_\a\cdot \psi^k\otimes \n_{e_\a}\n_{\f{\p}{\p t}}W_k+e_\a\cdot \psi^k\otimes R_{e_\a,\f{\p}{\p t}}W_k\\
&=e_\a\cdot \n_{e_\a}(\psi^k\otimes \n_{\f{\p}{\p t}}W_k)+e_\a\cdot \psi^k\otimes R_{e_\a,\f{\p}{\p t}}W_k\\
&=D\hskip -2.9mm \slash\ (\psi^k\otimes \n_{\f{\p}{\p t}}W_k)+e_\a\cdot \psi^k\otimes R^N_{\phi_*(e_\a),\xi}W_k.
\endaligned$$
Please note that here and in the following text $R_{XY}=-[\n_X,\n_Y]+\n_{[X,Y]}$.
In conjunction with (\ref{eq1}), we have
\begin{equation}\label{psi2}
\aligned
II&=\int_M \lan \f{d}{dt}\Big|_{t=0}\psi,D\hskip -2.9mm \slash\ \psi\ran+\lan \psi,\f{d}{dt}\Big|_{t=0}D\hskip -2.9mm \slash\ \psi\ran\\
&=\int_M \lan \psi,D\hskip -2.9mm \slash\ (\psi^k\otimes \n_{\f{\p}{\p t}}W_k)+e_\a\cdot \psi^k\otimes R^N_{\phi_*(e_\a),\xi}W_k\ran\\
&=\int_M \lan D\hskip -2.9mm \slash\ \psi,\psi^k\otimes \n_{\f{\p}{\p t}}W_k\ran+\int_M \lan \psi^j,e_\a\cdot \psi^k\ran\lan W_j,R_{\phi_*(e_\a),\xi}^N W_k\ran\\
&=\int_M \lan \psi^j,e_\a\cdot \psi^k\ran\overline{\lan R_{\phi_*(e_\a),\xi}^N W_k,W_j\ran}\\
&=-\int_M \lan \psi^j,e_\a\cdot \psi^k\ran\lan R_{\overline{W}_j,W_k}^N \phi_*(e_\a),\xi\ran
\endaligned
\end{equation}
Substituting (\ref{psi1}) and (\ref{psi2}) into (\ref{psi}) yields
\begin{equation}
\f{dL}{dt}=-\int_M \Big\lan 2\tau(\phi)+\lan \psi^j,e_\a\cdot \psi^k\ran R_{\overline{W}_j,W_k}^N \phi_*(e_\a),\xi\Big\ran.
\end{equation}
Thereby (\ref{eq4}) follows.

\end{proof}

\begin{rem}

The Euler-Lagrange equations of $L$ was firstly derived in \cite{c-j-l-w1}. But our denotation is different.

\end{rem}

\Section{Zeros of harmonic spinor fields}{Zeros of harmonic spinor fields}\label{s1}

In this section, $M$ and $N$ are both oriented Riemannian surfaces.
Then $\Si M=\Si^+ M\oplus \Si^- M$,
where
\begin{equation}\label{psi+}
\Si^\pm M=\{\xi\in \Si M: \sqrt{-1}e_1\cdot e_2\cdot \xi=\pm 1\}.
\end{equation}
(Here $\{e_1,e_2\}$ is an orthonormal basis of $T_{\pi(\xi)}M$, and $\pi$ denotes the bundle projection of $\Si M$ onto $M$.)
In conjunction with $(\phi^{-1}TN)^\C=\phi^{-1}(T^{(1,0)}N)\oplus \phi^{-1}(T^{(0,1)}N)$, we have
\begin{equation}\aligned
\Si M\otimes (\phi^{-1}TN)^\C&=\big(\Si^+M\otimes \phi^{-1}(T^{(1,0)}N)\big)\oplus \big(\Si^+M\otimes \phi^{-1}(T^{(0,1)}N)\big)\\
&\oplus\big(\Si^-M\otimes \phi^{-1}(T^{(1,0)}N)\big)\oplus\big(\Si^-M\otimes \phi^{-1}(T^{(0,1)}N)\big).
\endaligned
\end{equation}
Denote by $\pi_1^+,\pi_0^+,\pi_1^-,\pi_0^-$ the projections of $\Si M\otimes (\phi^{-1}TN)^\C$ onto the
subbundles, respectively. Let $X$ be a tangent vector field on $M$, then $\n_X$ keeps
$\G(\Si^\pm M)$, $\G\big(\phi^{-1}(T^{(1,0)}N)\big)$ and $\G\big(\phi^{-1}(T^{(0,1)}N)\big)$ invariant,
and $X\cdot \Si^\pm M\subset \Si^\mp M$; therefore
\begin{equation}
\aligned
&\Dirac\Big(\G\big(\Si^\pm M\otimes \phi^{-1}(T^{(1,0)}N)\big)\Big)\subset \G\big(\Si^\mp M\otimes \phi^{-1}(T^{(1,0)}N)\big),\\
&\Dirac\Big(\G\big(\Si^\pm M\otimes \phi^{-1}(T^{(0,1)}N)\big)\Big)\subset \G\big(\Si^\mp M\otimes \phi^{-1}(T^{(0,1)}N)\big).
\endaligned
\end{equation}
Hence $\Dirac \psi=0$ yields that $\pi_1^+(\psi), \pi_0^+(\psi), \pi_1^-(\psi), \pi_0^-(\psi)$ are all harmonic spinor fields along $\phi$.

Let $\psi\in \G\big(\Si^+ M\otimes \phi^{-1}(T^{(1,0)}N)\big)$ be harmonic, we shall derive the equation
of $\psi$ in local complex coordinates.

Let $z=x+\sqrt{-1}y$, $w=u+\sqrt{-1}v$ be complex coordinates of $M, N$, respectively. Then the metric of
$M, N$ are of the forms $\la(z)|dz|^2$, $\rho(w)|dw|^2$, respectively.
Denote
\begin{equation}
s=\{e_1,e_2\},\qquad \mbox{where } e_1=\la^{-\f{1}{2}}\f{\p}{\p x},\ e_2=\la^{-\f{1}{2}}\f{\p}{\p y}.
\end{equation}
then $s$ is a local tangent orthonormal frame bundle; i.e. $s$ is a smooth section of $SO U$, $U\subset M$.
Let $\td{s}\in \G(Spin U)$ be a lift of $s$, i.e. $\eta\circ \td{s}=s$. Denote
\begin{equation}\label{d1}
\psi^+=[\td{s},\si],\qquad \psi^-=e_1\cdot \psi^+,
\end{equation}
where $\si$ is a unit vector in $\Si_2^+$, then from (\ref{psi+}),
\begin{equation}\label{eq2}
\aligned
&e_1\cdot \psi^+=\psi^-,\qquad e_1\cdot \psi^-=-\psi^+,\\
&e_2\cdot \psi^+=\sqrt{-1}\psi^-,\qquad e_2\cdot \psi^-=\sqrt{-1}\psi^+.
\endaligned
\end{equation}
And furthermore,
\begin{equation}
\aligned
&\f{\p}{\p z}\cdot \psi^+=\la^{\f{1}{2}}\psi^-,\qquad \f{\p}{\p z}\cdot \psi^-=0,\\
&\f{\p}{\p \bar{z}}\cdot \psi^+=0,\qquad \f{\p}{\p \bar{z}}\cdot \psi^-=-\la^{\f{1}{2}}\psi^+.
\endaligned
\end{equation}
By the definition of the connection on $\Si M$, we
have
\begin{equation}
\aligned
\p\hskip -2.6mm \slash\ \psi^+&=e_\a\cdot \n_{e_\a}\psi^+=\f{1}{2}e_\a\cdot \lan \n_{e_\a}e_1,e_2\ran e_1\cdot e_2\cdot \psi^+\\
&=\f{1}{2}\lan \n_{e_1}e_1,e_2\ran e_1\cdot e_1\cdot e_2\cdot \psi^++\f{1}{2}\lan \n_{e_2}e_1,e_2\ran e_2\cdot e_1\cdot e_2\cdot \psi^+\\
&=\f{1}{2}\f{\p \la^{-\f{1}{2}}}{\p y}(-\sqrt{-1})\psi^--\f{1}{2}\f{\p \la^{-\f{1}{2}}}{\p x}\psi^-\\
&=-\f{\p \la^{-\f{1}{2}}}{\p \bar{z}}\psi^-.
\endaligned
\end{equation}
Let $f$ be a smooth function on $U$, such that $\psi=f \psi^+\otimes \f{\p}{\p w}$, then
\begin{equation}\label{Dirac1}
\aligned
0&=\Dirac \psi=
D\hskip -2.9mm \slash\ (f\psi^+\otimes \f{\p}{\p w})\\
&=\p\hskip -2.6mm \slash\ \psi^+\otimes f\f{\p}{\p w}+\f{2}{\la}\Big(\f{\p}{\p z}\cdot \psi^+\otimes \n_{\f{\p}{\p \bar{z}}}(f\f{\p}{\p w})+\f{\p}{\p \bar{z}}\cdot \psi^+\otimes \n_{\f{\p}{\p z}}(f\f{\p}{\p w})\Big)\\
&=-\f{\p \la^{-\f{1}{2}}}{\p \bar{z}}\psi^-\otimes f\pw+2\la^{-\f{1}{2}}\psi^-\otimes \Big(\pd{f}{\bar{z}}\pw+f\n_{\pd{w}{\bar{z}}\pw+\pd{\bar{w}}{\bar{z}}\f{\p}{\p \bar{w}}}^N \pw\Big)\\
&=2\la^{-\f{1}{2}}\Big(\f{1}{4}\pd{\log \la}{\bar{z}}f+\pd{\log \rho}{w}\pd{w}{\bar{z}}f+\pd{f}{\bar{z}}\Big)\psi^-\otimes \pw.
\endaligned
\end{equation}
Thereby we get the equation of $f$ as follows:
\begin{equation}\label{eq3}
\pd{f}{\bar{z}}+\Big(\f{1}{4}\pd{\log \la}{\bar{z}}+ \pd{\log\rho}{w}\pd{w}{\bar{z}}\Big)f=0.
\end{equation}
From it, we can prove the following proposition.

\begin{pro}
If $\psi$ is a harmonic spinor field along $\phi$, which is a smooth map between two oriented Riemannian surfaces, then
$|\pi_1^+(\psi)|$ is identically zero or it has isolated zeroes. So are $|\pi_0^+(\psi)|, |\pi_1^-(\psi)|$ and $|\pi_0^-(\psi)|$.

\end{pro}

\begin{proof}

As we have seen, $\pi_1^+(\psi)$ is harmonic whenever $\psi$ is harmonic. Denote
$\pi_1^+(\psi)=f\psi^+\otimes \pw$, then $f$ satisfies (\ref{eq3}). Denote
$$h(z)=\f{1}{4}\pd{\log \la}{\bar{z}}+ \pd{\log\rho}{w}\pd{w}{\bar{z}},$$
then $\pd{f}{\bar{z}}+hf=0$. Let $\ze$ be a local solution of $\pd{\ze}{\bar{z}}=h$,
then
$$\pd{(fe^\ze)}{\bar{z}}=-h fe^\ze+hf e^\ze=0.$$
i.e. $fe^\ze$ is holomorphic. Hence the conclusion follows from the well-known fact that the zeros of a holomorphic function are isolated, unless it is identically
zero. And the proof for $|\pi_0^+(\psi)|, |\pi_1^-(\psi)|$ and $|\pi_0^-(\psi)|$ is similar.

\end{proof}

\Section{Weitzenb\"{o}ck-type formulas and Bochner-type formulas}{Weitzenb\"{o}ck-type formulas and Bochner-type formulas}\label{s3}

For a spinor field $\psi$ along a map $\phi:M\ra N$, where $(M^m,h)$ is a Riemannian spin manifold, we can proceed
as \cite{c-j-l-w1} Proposition 3.4 to have the following Weitzenb\"{o}ck-type formula.
\begin{equation}
D\hskip -2.9mm \slash\ ^2\psi=-\td{\n}_{e_\a}\td{\n}_{e_\a}\psi+\f{1}{4}S\psi+\f{1}{2}\sum_{\a\neq \be}e_\be\cdot e_\a\cdot \psi^j\otimes R_{\phi_*(e_\a),\phi_*(e_\be)}^N W_j.
\end{equation}
Here $\psi=\psi^j\otimes W_j$, $\{e_\a\}$ is a local tangent orthonormal frame field on $M$ such that $\n e_\a=0$ at the considered point,
and $S$ is the scalar curvature. When $m=2$, $S=2K_M$, where $K_m$ denote the Gauss curvature of $M$, then
\begin{equation}\label{wei}
D\hskip -2.9mm \slash\ ^2\psi=-\td{\n}_{e_\a}\td{\n}_{e_\a}\psi+\f{1}{2}K_M\psi+e_2\cdot e_1\cdot \psi_j\otimes R_{\phi_*(e_1),\phi_*(e_2)}^N W_j.
\end{equation}

Now we assume $M, N$ are both oriented Riemannian surfaces; $\psi\in \G\big(\Si^+ M\otimes \phi^{-1}(T^{(1,0)}N)\big)$ and $D\hskip -2.9mm \slash\ \psi=0$. For arbitrary $x\in M$,
let $r>0$ such that $\exp_x:B(r)\ra M$ is injective; denote $U=\exp_x\big(B(r)\big)$, then we can define a local section of
$\Si^+ M$ (denoted by $\psi^+$) and a local section of $\phi^{-1}(T^{(1,0)}N)$ (denoted by $W$) on $U$, such that
for any geodesic $\g$ starting from $x$,
$$\n_{\dot{\g}}\psi^+=\n_{\dot{\g}}W=0,$$
and $\lan \psi^+,\psi^+\ran=1$, $\lan W,W\ran=1$;
hence at $x$,
$$\n_{e_\a}\psi^+=\n_{e_\a}\n_{e_\a}\psi^+=\n_{e_\a}W=\n_{e_\a}\n_{e_\a}W=0.$$
$\psi$ could locally be expressed by $\psi=f\psi^+\otimes W$, then at $x$,
\begin{equation}\label{wei1}
\td{\n}_{e_\a}\td{\n}_{e_\a}\psi=\De f\ \psi^+\otimes W.
\end{equation}
Denote $W=\f{\sqrt{2}}{2}(V_1-\sqrt{-1}V_2)$, then $V_2=J^N V_1$, where $J^N$ is the complex structure on $N$, and
$g(V_i,V_j)=\de_{ij}$. Denote
$$\phi_* e_i=\phi_{ij}V_j,$$
then
\begin{equation}\label{wei2}
\aligned
R_{\phi_* e_1,\phi_* e_2}^N W&=R_{\phi_{11}V_1+\phi_{12}V_2,\phi_{21}V_1+\phi_{22}V_2}^N\big(\f{\sqrt{2}}{2}(V_1-\sqrt{-1}V_2)\big)\\
&=\det(\phi_{ij})R_{V_1,V_2}^N\big(\f{\sqrt{2}}{2}(V_1-\sqrt{-1}V_2)\big)\\
&=\sqrt{-1}J(\phi)K_N\ W.
\endaligned
\end{equation}
Here $J(\phi)$ denotes the Jacobian of $\phi$ and $K_N$ denotes the Gauss curvature of $N$.
Substituting (\ref{wei1}) and (\ref{wei2}) into (\ref{wei}) yields
\begin{equation}\aligned
0&=D\hskip -2.9mm \slash\ ^2\psi=-\td{\n}_{e_\a}\td{\n}_{e_\a}\psi+\f{1}{2}K_M\psi+e_2\cdot e_1\cdot \psi^+\otimes R_{\phi_*e_1,\phi_* e_2}^N W\\
&=(-\De f+\f{1}{2}K_M f-K_N J(\phi)f)\ \psi^+\otimes W.
\endaligned
\end{equation}
i.e.
\begin{equation}
\De f=\f{1}{2}K_M f-K_N J(\phi)f\qquad \mbox{at }x.
\end{equation}
Furthermore,
\begin{equation}\label{Bochner1}
\aligned
\De |\psi|^2&=\De |f|^2=\bar{f}\De f+f\De \bar{f}+2|\n f|^2\\
&=K_M |\psi|^2-2K_N J(\phi)|\psi|^2+2|\n \psi|^2.
\endaligned
\end{equation}

From
\begin{equation}
\aligned
0&=e_1\cdot D\hskip -2.9mm \slash\ \psi=e_1\cdot e_1\cdot \n_{e_1}\psi+e_1\cdot e_2\cdot \n_{e_2}\psi\\
&=-(\n_{e_1}f)\psi^+\otimes W+(\n_{e_2}f)e_1\cdot e_2\cdot \psi^+\otimes W\\
&=-(\n_{e_1}f+\sqrt{-1}\n_{e_2}f)\psi^+\otimes W
\endaligned
\end{equation}
we have
\begin{equation}
\n_{\bar{Z}}f=0.
\end{equation}
Here $Z=\f{\sqrt{2}}{2}(e_1-\sqrt{-1}e_2)$ and $\bar{Z}=\f{\sqrt{2}}{2}(e_1+\sqrt{-1}e_2)$, which satisfy
$h(Z,\bar{Z})=1, h(Z,Z)=h(\bar{Z},\bar{Z})=0$. Then $\n f=(\n_Z f)\bar{Z}$ and
\begin{equation}
|\n \psi|^2=|\n f|^2=|\n_Z f|^2.
\end{equation}
Furthermore, from
\begin{equation}
\aligned
\n|\psi|^2&=\n|f|^2=f\n \bar{f}+\bar{f}\n f\\
&=f(\n_{\bar{Z}}\bar{f})Z+\bar{f}(\n_Z f)\bar{Z}
\endaligned
\end{equation}
we arrive at
\begin{equation}\label{Bochner2}
\big|\n|\psi|^2\big|^2=2|f|^2|\n_Z f|^2=2|\psi|^2|\n \psi|^2.
\end{equation}
Substituting (\ref{Bochner2}) into (\ref{Bochner1}) yields
\begin{equation}
\De|\psi|^2=K_M |\psi|^2-2K_N J(\phi)|\psi|^2+\f{\big|\n|\psi|^2\big|^2}{|\psi|^2}.
\end{equation}
And at last we derive the following Bochner-type formula
\begin{equation}
\De \log|\psi|=\f{1}{2}K_M-K_N J(\phi).
\end{equation}
Similarly, when $\psi\in \Si^+M\otimes \phi^{-1}(T^{(0,1)}N)$, $\Si^-M\otimes \phi^{-1}(T^{(1,0)}N)$ or $\Si^-M\otimes \phi^{-1}(T^{(0,1)}N)$,
the corresponding Bochner-type formulas could be derived. We write those results as the following theorem.

\begin{thm}
Let $M$ and $N$ are both oriented Riemannian surfaces.
If $\psi$ is a harmonic spinor field along $\phi:M\ra N$, then $\log\big|\pi_1^+(\psi)\big|$, $\log\big|\pi_0^+(\psi)\big|$,
$\log\big|\pi_1^-(\psi)\big|$ and $\log\big|\pi_0^-(\psi)\big|$ satisfy Bochner-type formulas as follows:
\begin{eqnarray}
&&\De \log\big|\pi_1^+(\psi)\big|=\f{1}{2}K_M-K_N J(\phi),\label{Bochner3}\\
&&\De \log\big|\pi_0^+(\psi)\big|=\f{1}{2}K_M+K_N J(\phi),\\
&&\De \log\big|\pi_1^-(\psi)\big|=\f{1}{2}K_M+K_N J(\phi),\\
&&\De \log\big|\pi_0^-(\psi)\big|=\f{1}{2}K_M-K_N J(\phi).\label{Bochner4}
\end{eqnarray}
\end{thm}

When $M$ and $N$ are both compact, since the zeros of $|\pi_1^+(\psi)|$ are isolated, there exist a finite number of
zeros $p_1,\cdots,p_k$ in $M$. And similarly the zeros of $|\pi_0^+(\psi)|$, $|\pi_1^-(\psi)|$ and $|\pi_0^-(\psi)|$
are finite.
Integrating both side of (\ref{Bochner3})-(\ref{Bochner4}) on $M$, in conjunction with divergence theorem and Gauss-Bonnet formula, we can proceed as
\cite{s-y}pp. 11-12 to get the proposition:

\begin{thm}\label{t1}
Let $M$ and $N$ are both compact oriented Riemannian surfaces, $\psi$ is a harmonic spinor field along $\phi:M\ra N$.
If $|\pi_1^+(\psi)|$ is not identically zero,
then
\begin{equation}
\sum_{p\in M,|\pi_1^+{\psi}|(p)=0}n_p^+=g_M-1-\deg(\phi)(2g_N-2).
\end{equation}
If $|\pi_0^+(\psi)|$ is not identically zero,
then
\begin{equation}
\sum_{p\in M,|\pi_0^+{\psi}|(p)=0}m_p^+=g_M-1+\deg(\phi)(2g_N-2).
\end{equation}
If $|\pi_1^-(\psi)|$ is not identically zero,
then
\begin{equation}
\sum_{p\in M,|\pi_1^-{\psi}|(p)=0}n_p^-=g_M-1+\deg(\phi)(2g_N-2).
\end{equation}
If $|\pi_0^-(\psi)|$ is not identically zero,
then
\begin{equation}
\sum_{p\in M,|\pi_0^-{\psi}|(p)=0}m_p^-=g_M-1-\deg(\phi)(2g_N-2).
\end{equation}
Here $n_p^+, m_p^+, n_p^-, m_p^-$ are respectively the order of $|\pi_1^+(\psi)|$, $|\pi_0^+(\psi)|$, $|\pi_1^-(\psi)|$,
$|\pi_0^-(\psi)|$ at $p$; $\deg(\phi)$ denotes the degree of mapping; $g_M$ and $g_N$ are genus of $M$ and $N$, respectively.

\end{thm}

\Section{Proof of main theorems}{Proof of main theorems}\label{s4}

In conjunction with Proposition \ref{p1} and Theorem \ref{t1}, it is not difficult to obtain:

\begin{thm}\label{t2}

$M$ and $N$ are both compact oriented Riemannian surfaces, and $(\phi,\psi)$ is a Dirac-harmonic map from $M$ to $N$.
If $g_M=0$ or $|g_M-1|<|\deg(\phi)||2g_N-2|$, then $\phi$ has to be a harmonic map.

\end{thm}

\begin{proof}

If $g_M=0$ or $|g_M-1|<|\deg(\phi)||2g_N-2|$, then $g_M-1-\deg(\phi)(2g_N-2)<0$ or $g_M-1+\deg(\phi)(2g_N-2)<0$. Hence
from Theorem \ref{t1}, either $\pi_1^+(\psi)=\pi_0^-(\psi)=0$ or $\pi_0^+(\psi)=\pi_1^-(\psi)=0$ could be obtained. If
$\pi_1^+(\psi)=\pi_0^-(\psi)=0$, then there exist smooth functions $f$ and $g$, such that
$$\psi=f\psi^+\otimes \pdw+g\psi^-\otimes \pw,$$
where the definition of $\psi^+,\psi^-,\pw,\pdw$ is similar to Section \ref{s1}.
Hence by (\ref{eq4}),
\begin{equation}
\aligned
\tau(\phi)=&-\f{1}{2}\lan f\psi^+,e_\a\cdot f\psi^+\ran R_{\pw,\pdw}^N\phi_*(e_\a)-\f{1}{2}\lan g\psi^-,e_\a\cdot g\psi^-\ran R_{\pdw,\pw}^N \phi_*(e_\a)\\
&-\f{1}{2}\lan f\psi^+,e_\a\cdot g\psi^-\ran R_{\pw,\pw}^N \phi_*(e_\a)-\f{1}{2}\lan g\psi^-,e_\a\cdot f\psi^+\ran R_{\pdw,\pdw}\phi_*(e_\a)\\
=&0.
\endaligned
\end{equation}
So $\phi$ is a harmonic map. When $\pi_0^+(\psi)=\pi_1^-(\psi)=0$, the proof is similar.

\end{proof}

It is well known that $S^2$ is biholomorphically isomorphic to $\C^*=\C\cup \{\infty\}$; hence in the following text
we identify $S^2$ and $\C^*$.
Let $h=\la(z)|dz|^2$ be a metric on $\C^*$; denote $\td{z}=z^{-1}$, then
$$\la(z)|dz|^2=\la(\td{z}(z))|\td{z}|^{-4}d\td{z}^2;$$
hence $\la(\td{z}(z))|\td{z}|^{-4}$ is regular at $\td{z}=0$; then there exists a constant $c>0$, such that
\begin{equation}\label{metric}
\lim_{z\ra \infty} \la(z)|z|^4=c.
\end{equation}
And the definition of $e_1,e_2,\psi^+,\psi^-$ is similar to Section \ref{s1}.

\begin{thm}

If $M=S^2=\C^*$ equipped with metric $h=\la(z)|dz|^2$ and $N=S^2$ equipped with arbitrary metric, $(\phi,\psi)$ is a nontrivial Dirac-harmonic map from $M$ to $N$, then
$\phi$ has to be holomorphic or anti-holomorphic, $\psi$ could be written in the form
\begin{equation}
\psi=e_\a\cdot \Psi\otimes \phi_*(e_\a);
\end{equation}
and there exist two meromorphic function $u_1,u_2$ on $\C^*$ such that
\begin{equation}\label{twistor}
\Psi=\bar{u}_1\la^{\f{1}{4}}\psi^+ +u_2\la^{\f{1}{4}}\psi^-;
\end{equation}
if $u_i$($i=1$ or $2$) has a pole of order $k$ at $z_0\in \C$,
then $|d\phi|(z_0)=0$ and the order of $|d\phi|$ at $z_0$ is no less than $k$; if $\infty$ is a pole of order $k\geq 2$,
then $|d\phi|(\infty)=0$ and the order of $|d\phi|$ at $\infty$ is no less than $k-1$. And vice versa.
\end{thm}

\begin{proof}

By Theorem \ref{t2}, $\phi$ has to be a harmonic map. It is well known that when $\deg \phi=0$, $\phi$ is a constant mapping;
when $\deg \phi\geq 1$, $\phi$ is holomorphic; and when $\deg\phi\leq -1$, $\phi$ is anti-holomorphic (cf. \cite{s-y}pp. 11-12).
From Theorem \ref{t1}, when $\deg \phi=0$, $|\pi_1^+(\psi)|,|\pi_0^+(\psi)|, |\pi_1^-(\psi)|,|\pi_0^-(\psi)|$ are all identically zero,
hence $\psi=0$; it is a trivial solution of (\ref{eq1})-(\ref{eq4}). When $\deg \phi\geq 1$, we have $|\pi_0^+(\psi)|=|\pi_1^-(\psi)|=0$; since $\phi$ is holomorphic,
$\phi_*(e_1)-\sqrt{-1}\phi_*(e_2)\in T^{(1,0)}N$, $\phi_*(e_1)+\sqrt{-1}\phi_*(e_2)\in T^{(0,1)}N$, there exist two functions
$f, g$ (possibly with isolated singularities), such that
\begin{equation}
\psi=f\psi^+\otimes (\phi_*(e_1)-\sqrt{-1}\phi_*(e_2))+g\psi^-\otimes(\phi_*(e_1)+\sqrt{-1}\phi_*(e_2)).
\end{equation}
From (\ref{eq2}), it is easy to obtain
\begin{equation}\label{psi3}
\psi=e_\a\cdot \Psi\otimes \phi_*(e_\a);
\end{equation}
where $\Psi=g\psi^+-f\psi^-$.
When $\deg \phi \leq -1$, similarly we can construct a spinor $\Psi$ (possibly with isolated singularities) satisfying (\ref{psi3}).

By \cite{c-j-l-w1} Proposition 2.2, $\Psi$ is a twistor spinor, i.e.
\begin{equation}\label{eq5}
\n_v \Psi+\f{1}{2}v\cdot \dirac\Psi=0
\end{equation}
for any $v\in T_p S^2$, where $p$ is an arbitrary regular point of $\Psi$.
If $\Psi\in \G(\Si^+ M)$, then from (\ref{psi+}), (\ref{eq5}) is equivalent to
\begin{equation}
\n_{\pz}\Psi=0.
\end{equation}
Denote
\begin{equation}
\Psi_0^+=\la^{\f{1}{4}}\psi^+,
\end{equation}
then from
\begin{equation}
\aligned
\n_{\pz}\psi^+&=\f{1}{2}\lan \n_{\pz}e_1,e_2\ran e_1\cdot e_2\cdot \psi^+\\
&=-\f{1}{4}\sqrt{-1}\la^{\f{1}{2}}\lan \n_{e_1-\sqrt{-1}e_2}e_1,e_2\ran \psi^+\\
&=-\f{1}{4}\sqrt{-1}\la^{\f{1}{2}}\Big(\pd{(\la^{-\f{1}{2}})}{y}+\sqrt{-1}\pd{(\la^{-\f{1}{2}})}{x}\Big)\psi^+\\
&=-\f{1}{4}\pd{\log\la}{ z}\psi^+,
\endaligned
\end{equation}
we have
\begin{equation}
\n_{\pz}\Psi_0^+=\f{1}{4}\pd{\log \la}{z}\la^{\f{1}{4}}\psi^+-\f{1}{4}\pd{\log \la}{z}\la^{\f{1}{4}}\psi^+=0.
\end{equation}
Let $u_1$ be a function on $C^*$ (possibly with isolated singularities) such that
$$\Psi=\bar{u}_1\Phi_0^+,$$
then it is easy to obtain $\pd{u_1}{\bar{z}}=0$. Similarly, if $\Psi\in \G(\Si^- M)$ is a twistor spinor,
then we could obtain
$$\Psi=u_2\Phi_0^-,$$
where $u_2$ is a meromorphic function and $\Phi_0^-=\la^{\f{1}{4}}\psi^-$. Thereby (\ref{twistor}) follows.
The last statement is followed from $\la>0$ on $\C$ and (\ref{metric}).

On the other hand, let $\psi=e_\a\cdot \Psi\otimes \phi_*(e_\a)$, where
$\phi$ is holomorphic or anti-holomorphic, and $\Psi$ is a twistor spinor (possibly with isolated singularities) satisfying
(\ref{twistor}). Denote by $z_1,\cdots,z_l\in \C^*$ the singularities of $\Psi$,
 then it is easily to check that $(\phi,\psi)$ satisfies (\ref{eq1}) and (\ref{eq4}) on $\C^*-\{z_1,\cdots,z_l\}$, and by the assumption on the poles of $u_i$ and the zeros
 of $|d\phi|$,
$|\psi|$ is bounded on $\C^*-\{z_1,\cdots,z_l\}$. Hence the energy of $(\phi,\psi)$ on $\C^*$
$$E(\phi,\psi,\C^*)=\int_{\C^*} (|d\phi|^2+|\psi|^4)$$
is finite. The removable singularity theorem (see \cite{c-j-l-w1}) yields that $(\phi,\psi)$ is a Dirac-harmonic map
from $M$ to $N$.

\end{proof}

ACKNOWLEDGEMENT.\ The author wishes to express his sincere gratitude to Professor Y.L. Xin in Fudan University, for his inspiring suggestions.

\bibliographystyle{amsplain}

\begin{thebibliography}{10}

\bibitem{c-j-l-w1} Qun Chen, J\"{u}rgen Jost, Jiayu Li and Guofang Wang: Dirac-harmonic maps.
Math. Z. {\bf 254} (2006), 409-432.

\bibitem{c-j-l-w2} Qun Chen, J\"{u}rgen Jost, Jiayu Li and Guofang Wang: Regularity and energy identities
for Dirac-harmonic maps. Math. Z. {\bf 251} (2005), 61-84.

\bibitem{c-j-w} Qun Chen, J\"{u}rgen Jost and Guofang Wang: Liouville theorems for Dirac-harmonic maps.
J. Math. Phys. {\bf 48} (2007), 113517, 13pages.

\bibitem{f} T. Friedrich: Dirac operators in Riemannian geometry. Graduate Studies in
Mathematics. {\bf 25}. American Mathematical Society, Providence, RI, (2000), xvi+195 pp.

\bibitem{l-m} H.B. Lawson and M.-L. Michelsohn: Spin geometry. Princeton University Press,
Princeton, 1989.

\bibitem{s-y} R. Schoen, S. T. Yau: Lectures on harmonic maps. Conference Proceedings and Lecture Notes in Geometry and Topology, II. International Press,
1997.

\bibitem{x} Y. L. Xin: Geometry of harmonic maps. Birkh\"{a}user, 1996.

\bibitem{z1} Liang Zhao: Energy identities for Dirac-harmonic maps. Calc. Var. PDE {\bf 28} (2007), 121-138.

\bibitem{z2} Miaomiao Zhu: Dirac-harmonic maps from degenerating spin surfaces I: the Neveu-Schwarz case. arxiv: 0803. 3723.


\end{thebibliography}

\end{document}